\documentclass[preprint, 10pt, english]{elsarticle}
\usepackage{amsthm}
\usepackage{amsmath}
\usepackage{latexsym, amssymb}
\usepackage{txfonts}
\usepackage{mathtools}
\usepackage{color}
\usepackage{babel}
\usepackage[all]{xy}
\usepackage[capitalise]{cleveref}

\newtheorem{thm}{Theorem}[section] 

\newtheorem{cor}[thm]{Corollary}

\newtheorem{lem}[thm]{Lemma}
\newtheorem{prop}[thm]{Proposition}

\theoremstyle{definition}
\newtheorem{rem}[thm]{Remark}

\newcommand\operA[2]{{\if!#2!\operatorname{#1}\else{\operatorname{#1}_{#2}^{\phantom{I}}}\fi}} 

\newcommand{\Trace}[1][]{\if!#1!\operatorname{Tr}\else{\operatorname{Tr}_{#1}^{\phantom{I}}}\fi} 

\long\def\forget#1\forgotten{{}} %

\def\({\left(}
\def\){\right)}

\newcommand\MPf[1]{{\left<\left<{#1}\right]\right]}} 

\newcommand\LAY[3][]{{\begin{array}{c}\mbox{#2} \if#1!{}\else{+}\fi \\ \mbox{#3}\end{array}}}

\makeatletter

\def\ps@pprintTitle{%
 \let\@oddhead\@empty
 \let\@evenhead\@empty
 \def\@oddfoot{}%
 \let\@evenfoot\@oddfoot}

\newcommand{\bigperp}{%
  \mathop{\mathpalette\bigp@rp\relax}%
  \displaylimits
}

\newcommand{\bigp@rp}[2]{%
  \vcenter{
    \m@th\hbox{\scalebox{\ifx#1\displaystyle2.1\else1.5\fi}{$#1\perp$}}
  }%
}
\makeatother

\newcommand{\Om}{\Omega}

\newcommand{\wg}{\wedge}
\newcommand{\bwg}{\bigwedge}

\renewcommand{\geq}{\geqslant}
\renewcommand{\leq}{\leqslant}

\DeclareMathOperator{\coker}{coker}

\newif\iffurther
\furtherfalse

\journal{??}

\begin{document}
\begin{frontmatter}

\title{The $u^n$-invariant and the Symbol Length of $H_2^n(F)$}

\author{Adam Chapman}
\ead{adam1chapman@yahoo.com}
\address{Department of Computer Science, Tel-Hai Academic College, Upper Galilee, 12208 Israel}
\author{Kelly McKinnie}
\ead{kelly.mckinnie@mso.umt.edu}
\address{Department of Mathematics, University of Montana, Missoula, MT 59812, USA}

\begin{abstract}
Given a field $F$ of $\operatorname{char}(F)=2$, we define $u^n(F)$ to be the maximal dimension of an anisotropic form in $I_q^n F$. For $n=1$ it recaptures the definition of $u(F)$. We study the relations between this value and the symbol length of $H_2^n(F)$, denoted by $sl_2^n(F)$. We show for any $n \geq 2$ that if $2^n \leq u^n(F) \leq u^2(F) < \infty$ then $sl_2^n(F) \leq \prod_{i=2}^n (\frac{u^i(F)}{2}+1-2^{i-1})$. As a result, if $u(F)$ is finite then $sl_2^n(F)$ is finite for any $n$, a fact which was previously proven when $\operatorname{char}(F) \neq 2$ by Saltman and Krashen. We also show that if $sl_2^n(F)=1$ then $u^n(F)$ is either $2^n$ or $2^{n+1}$.
\end{abstract}

\end{frontmatter}

\section{Introduction}

The Kato-Milne cohomology groups $H_p^n(F)$ of a field $F$ of $\operatorname{char}(F)=p$ are generated by decomposable differential forms. The symbol length of an element in $H_p^n(F)$ is the minimal number of decomposable differential forms required to express it. The symbol length $sl_p^n(F)$ of $H_p^n(F)$ is the supremum on the symbol lengths of all its elements.
When $\operatorname{char}(F)=2$, $H_2^n(F) \cong I_q^n F/I_q^{n+1} F$ for any $n$ by \cite{Kato:1982}, where $I_q^n F$ is the subgroup of $W_q F$ generated by scalar multiples of quadratic $n$-fold Pfister forms (see \cite{EKM} for reference.)
The $u$-invariant of $F$, denoted by $u(F)$, is defined to be the maximal dimension of nonsingular anisotropic forms over $F$. Let $u^n(F)$ denote the maximal dimension of an anisotropic form in $I_q^n F$. It follows that $u^1(F)=u(F)$ and for any $i<j$, $u^i(F) \geq u^j(F)$.

For $n=1$,  $H^1_2(F) \cong F/\wp(F)$, and so $sl_2^1(F)=1$. For $n=2$, $H^2_2(F) \cong \prescript{}{2}Br(F)$ and $sl_2^2(F)$ is the symbol length of central simple $F$-algebras of exponent 2. By \cite[Corollary 4.2]{Chapman:2017}, $sl^2_2(F) \leq \frac{u(F)}{2}-1$.
In this paper, we prove for any $n \geq 2$ that if $2^n \leq u^n(F) \leq u^2(F) < \infty$ then $sl_2^n(F) \leq \prod_{i=2}^n (\frac{u^i(F)}{2}+1-2^{i-1})$, and as a result, if $u(F)$ is finite then $sl_2^n(F)$ is finite for any $n$.
The equivalent result in the case of $\operatorname{char}(F) \neq 2$ was recently obtained in \cite{Saltman-higher-degrees} and \cite{Krashen:2016}.
We also show that if $n \geq 2$ and $sl_2^n(F)=1$ then $u^n(F)$ is either $2^n$ or $2^{n+1}$ and relate it to formerly obtained results from \cite{ElmanLam:1973}, \cite{Baeza:1982}, \cite{Kashima:1996} and \cite{ChapmanDolphin:2017}.

\section{Preliminaries}
Given a prime number $p$ and a field $F$ of $\operatorname{char}(F)=p$, $\Om_F^1$ is the space of absolute differential forms defined to be the
$F$-vector space generated by the symbols $da$ subject to the relations $d(a+b)=da+db$ and $d(ab)=adb+bda$ for any $a,b \in F$.
The space of $n$-differential forms $\Om_F^n$ for any positive integer $n$ is then defined by the
$n$-fold exterior power 
$\Om_F^n=\bwg^n(\Om_F^1)$, which is consequently an $F$-vector space spanned
by  $da_1\wg\ldots\wg da_n$, $a_i\in F$. The derivation $d$ extends to an 
operator $d\,:\,\Om_F^n \to \Om_F^{n+1}$ by $d(a_0da_1\wg\ldots\wg da_n)=
da_0\wg da_1\wg\ldots\wg da_n$.  We define $\Om_F^0=F$, $\Om_F^n=0$ for $n<0$, and
$\Om_F=\bigoplus_{n\geq 0}\Om_F^n$,  the algebra of differential forms
over $F$ with multiplication naturally defined by
$$(a_0da_1\wg\ldots\wg da_n)(b_0db_1\wg\ldots\wg db_m)=
a_0b_0da_1\wg\ldots\wg da_n\wg db_1\wg\ldots\wg db_m\,.$$ 

There exists a well-defined group homomorphism $\Om_F^n\to \Om_F^n/d\Om_F^{n-1}$, the
Artin-Schreier map $\wp$, which acts on decomposable differential forms as follows:
$$\alpha\frac{d \beta_1}{\beta_1}\wg\ldots\wg \frac{d \beta_n}{\beta_n}\,\longmapsto\,
(\alpha^p-\alpha)\frac{d \beta_1}{\beta_1}\wg\ldots\wg \frac{d \beta_n}{\beta_n}.$$
The group $H_p^{n+1}(F)$ is defined to be $\coker(\wp)$.

By \cite[Section 9.2]{GilleSzamuely:2006}, when $n=2$, there exists an isomorphism 
\begin{eqnarray*}
H_p^2(F) &{\longrightarrow}& \prescript{}{p} Br(F), \enspace \text{given by}\\
\alpha \frac{d\beta}{\beta} &\longmapsto & [\alpha,\beta)_{p,F},
\end{eqnarray*}
where $\prescript{}{p} Br(F)$ is the $p$-torsion part of the group of Brauer classes of central simple algebras over $F$, and $[\alpha,\beta)_{p,F}$ is the degree $p$ cyclic $p$-algebra 
$$F \langle x,y : x^p-x=\alpha, y^p=\beta, y x y^{-1}=x+1 \rangle.$$

For $p=2$, the groups $H_2^n(F)$ are strongly connected to the theory of quadratic forms. 
Recall that any nonsingular quadratic form can be written as $\beta_1 [1,\alpha] \perp \dots \perp \beta_n [1,\alpha_n]$ for some $\alpha_1,\dots,\alpha_n \in F$ and $\beta_1,\dots,\beta_n \in F^\times$, where $[1,\alpha]$ stands for the two-dimensional form $u^2+uv+\alpha_i v^2$, and $\perp$ is the orthogonal sum. We call the form $[1,\alpha]$ a quadratic 1-fold Pfister form and denote it by $\langle \langle \alpha]]$.
For any integer $n \geq 2$, we define the quadratic $n$-fold Pfister form $\MPf{ \beta_{n-1},\dots,\beta_{1},\alpha}$ recursively to be $\MPf{ \beta_{n-2},\dots,\beta_{1},\alpha} \perp \beta_{n-1} \MPf{ \beta_{n-2},\dots,\beta_{1},\alpha}$.
We denote by $P_n(F)$ the set of $n$-fold Pfister forms.
A quadratic Pfister form is isotropic if and only if it is hyperbolic.
We define $I_q^n F$ to be the subgroup of $I_q^1 F$ generated by the scalar multiples of quadratic $n$-fold Pfister forms. 
By \cite{Kato:1982}, there exists an epimorphism
\begin{eqnarray*}
e^{n+1} : I_q^{n+1}(F) &{\longrightarrow} & H_2^{n+1}(F), \enspace \text{given by}\\
 \langle \langle \beta_1,\dots,\beta_n,\alpha]]
& \longmapsto &  \alpha \frac{d \beta_1}{\beta_1}\wg\ldots\wg \frac{d \beta_n}{\beta_n} 
\end{eqnarray*}
with $\ker(e^{n+1})=I_q^{n+2}(F)$.
These maps are called the ``cohomological invariants".
The only two cohomological invariants that have an explicit formulation are $e^1$ and $e^2$.
The invariant $e^1$ is known as the ``Arf invariant", denoted by $\operatorname{Arf}(\dots)$. Given a form $\varphi=b_1 [1,a_1] \perp \dots \perp b_n [1,a_n]$ in $I_q F$, its Arf invariant can be computed using the formula
$\operatorname{Arf}(\varphi)=\sum_{i=1}^n a_i$.
Given any form $\varphi=b_1[1,a_1] \perp \dots \perp b_n [1,a_n]$ in $I_q F$, its Clifford algebra is computed using the formula
$C\ell(\varphi)=\bigotimes_{i=1}^n [a_i,b_i)_{2,F}$. The second cohomological invariant $e^2$ is actually the restriction $e^2=C\ell|_{I_q^2 F}$ and therefore is nicknamed the ``Clifford invariant".

\section{Symbol Length of $H_2^n(F)$}

In this section we provide an explicit upper bound for $sl_2^n(F)$ given $u^2(F),\dots,u^n(F)$. The main tool we use is Aravire and Laghribi's description of the Witt kernel of a multi-quadratic purely inseparable field extension \cite[Proposition 2]{AravireLaghribi:2013}. This method only works in characteristic 2 because of the existence of purely inseparable quadratic splitting fields.
We start with the following observation on fields $F$ of $\operatorname{char}(F)=p$ and finite $p$-rank:

\begin{rem}\label{prank} Let $F$ be a field with finite $p$-rank $m$ (i.e. $[F:F^p]=p^m$). Then for any $n$, $sl_p^{n+1}(F) \leq \binom{m}{n}$.
\label{l1}
\end{rem}

\begin{proof} Since $F$ has $p$-rank $m$, there is a differential basis $\{da_1,\ldots,da_m\}$ of $\Omega^1_F$ and $\Omega^n_F$ has differential basis $\{da_{j_1}\wedge\cdots\wedge da_{j_n}\}_{1\leq j_1<j_2<\cdots<j_n\leq m}$. There are $\binom mn$ elements in this basis, showing that each element in $\Omega^n_F$, and therefore also in $H^{n+1}_p(F)$, can be written as a sum of at most $\binom mn$ symbols.
\end{proof}

\begin{prop}[{\cite[Proposition 2]{AravireLaghribi:2013}}]\label{LagKing}
Let $F$ be a field of $\operatorname{char}(F)=2$, $b_1,\dots,b_\ell \in F^\times$ and let $K=F[\sqrt{b_1},\dots,\sqrt{b_\ell}]$ be a field extension of $F$. If a class in $H_2^n(F)$ has a trivial restriction to $H_2^n(K)$ then it can be written as $\sum_{i=1}^\ell \omega_i \wedge \frac{d b_i}{b_i}$ where $\omega_1,\dots,\omega_\ell \in H_2^{n-1}(F)$.
\end{prop}

\begin{lem}\label{Split}
Let $n \geq 2$ and $\varphi=b_1 [1,a_1] \perp \dots \perp b_{m-1} [1,a_{m-1}] \perp [1,a_1+\dots+a_{m-1}]$ be an anisotropic form of dimension $2m$ in $I_q^n F$ where $2m \geq 2^n$.
Then $\varphi_K$ is hyperbolic for $K=F[\sqrt{b_1},\dots,\sqrt{b_\ell}]$ where $\ell=m+1-2^{n-1}$.
\end{lem}

\begin{proof}
Since for each $i \in \{1,\dots,\ell\}$, $b_i$ is a square in $K$ and $[1,a_1] \perp \dots \perp [1,a_\ell] \perp [1,a_1+\dots+a_{m-1}] \sim_{Witt} [1,a_{\ell+1}+\dots+a_{m-1}]$, we have $\varphi_K \sim_{Witt} b_{\ell+1} [1,a_{\ell+1}] \perp \dots \perp [1,a_{m-1}] \perp [1,a_{\ell+1}+\dots+a_{m-1}]$. The latter is a form of dimension $2m-2(m+1-2^{n-1})=2^n-2$ in $I_q^n K$, hence it is hyperbolic by the Hauptsatz theorem \cite[Theorem 23.8]{EKM}. 
\end{proof}

\begin{thm}\label{GoodBound}
Let $F$ be a field of $\operatorname{char}(F)=2$.
Then for every $n \geq 2$, if $2^n \leq u^n(F) < \infty$ then $sl_2^n(F) \leq (\frac{u^n(F)}{2}+1-2^{n-1}) sl_2^{n-1}(F)$. Consequently, if in addition $u^2(F) < \infty$ then $sl_2^n(F) \leq \prod_{i=2}^n (\frac{u^i(F)}{2}+1-2^{i-1})$.
\end{thm}

\begin{proof}
Consider a class $C$ in $H_2^n(F)$.
It can be represented by an anisotrpic form $\varphi$ of dimension $2m$ in $I_q^n F$ where $m \leq \frac{u^n(F)}{2}$.
Since $n \geq 2$, $\varphi$ is in $I_q^2 F$ and can be written as $b_1 [1,a_1] \perp \dots \perp b_{m-1} [1,a_{m-1}] \perp c [1,a_1+\dots+a_{m-1}]$ for some $b_1,\dots,b_{m-1},c \in F^\times$ and $a_1,\dots,a_{m-1} \in F$. Since $c^{-1} \varphi$ has the same class in $H_2^n(F)$ as $\varphi$, assume without loss of generality that $c=1$.
By Lemma \ref{Split}, the restriction of $\varphi$ to $K=F[\sqrt{b_1},\dots,\sqrt{b_\ell}]$ is hyperbolic when $\ell=m+1-2^{n-1}$.
Therefore the restriction of $C$ to $H_2^n(K)$ is trivial, and by Proposition \ref{LagKing}, $C$ can be written as $\sum_{i=1}^\ell \omega_i \wedge \frac{d b_i}{b_i}$ where $\omega_1,\dots,\omega_\ell \in H_2^{n-1}(F)$. The symbol length of each $\omega_i$ is at most $sl_2^{n-1}(F)$. Therefore, the symbol length of $C$ is at most $\ell \cdot sl_2^{n-1}(F)$. As a result $sl_2^n(F) \leq (\frac{u^n(F)}{2}+1-2^{n-1}) sl_2^{n-1}(F)$. By iterating this fact $n-1$ times, we obtain $sl_2^n(F) \leq \left(\prod_{i=2}^n (\frac{u^i(F)}{2}+1-2^{i-1})\right) sl_2^1(F)$, but $sl_2^1(F)=1$, so $sl_2^n(F) \leq \prod_{i=2}^n (\frac{u^i(F)}{2}+1-2^{i-1})$.
\end{proof}

Note that since $u^n(F) \leq u(F)$ for any positive integer $n$, Theorem \ref{GoodBound} implies $sl_2^n(F) \leq \prod_{i=2}^n (\frac{u(F)}{2}+1-2^{i-1})$ for any $n \geq 2$ assuming $2^n \leq u(F) < \infty$. There exist fields with finite $u$-invariant and infinite $2$-rank (see \cite{MammoneTignolWadsworth:1991}), therefore Theorem \ref{GoodBound} gives a bound on the symbol length in cases not covered by Remark \ref{l1}. 
For $sl_2^3(F)$ we obtain a different upper bound which involves only $u^3(F)$ and not $u^2(F)$ by considering the Severi-Brauer variety and the Clifford invariant in a similar way to \cite{Saltman-higher-degrees}.

\begin{rem}[{\cite[Proposition 13.1 \& Remark 13.2]{BabicChernousov:2015}}]\label{BC}
Given a field $F$ of $\operatorname{char}(F)=2$ containing the algebraic closure $F_0$ of $\mathbb{F}_2$, every quadratic form of dimension $2m$ and trivial Arf invariant descends to a $(2m)$-dimension form of trivial Arf invariant over a subfield $K$ of $F$ of transcendence degree at most $m+1$ over $F_0$.
\end{rem}

\begin{thm}\label{main}
Let $F$ be a field with $\operatorname{char}(F)=2$ with finite $u^3(F)$.
Then $sl_2^3(F) \leq \binom{2m+2^{m-1}}{2}$ where $u^3(F)=2m$.
If we assume further that $F$ contains the algebraic closure $F_0$ of $\mathbb{F}_2$ then $sl_2^3(F) \leq \binom{m+1+2^{m-1}}{2}$.
\end{thm}

\begin{proof}
Let $\varphi$ be a quadratic form of dimension $2 m$ in $I_q^3 F$. Similarly to what was done in \cite{Saltman-higher-degrees}, there exists a subfield $K$ of $F$ of transcendence degree $\leq 2m$ over $\mathbb{F}_2$ such that $\varphi=\psi_F$ for some $(2m)$-dimensional form $\psi$ in $I_q^2 K$. (Note that the Arf invariant is trivialized by a suitable quadratic extension which does not affect the transcendence degree.)
If $F$ contains the algebraic closure $F_0$ of $\mathbb{F}_2$ then by Remark \ref{BC} we can assume the transcendence degree of $K$ over $F_0$ is at most $m+1$.

Since the Clifford algebra of $\psi$ is a central simple algebra over $K$ of index at most $2^{m-1}$, the Severi-Brauer variety of the underlying division algebra is of dimension at most $2^{m-1}-1$ as a projective variety by Chatelet's theorem (see \cite[5.3.6]{GilleSzamuely:2006}).
Since the Severi-Brauer variety is generic, there exists a specialization $L$ of its function field which is a subfield of $F$ and contains $K$. The transcendence degree of $L$ over $K$ is at most $2^{m-1}$ and $\psi_L$ is in $I_q^3 L$.
Then $L$ is of transcendence degree at most $\operatorname{tr.deg}(K)+2^{m-1}$ over $F_0$, which means its 2-rank is at most $\operatorname{tr.deg}(K)+2^{m-1}$.
Therefore the symbol length of the class of $\psi_L$ in $H_2^3(L)$ is $\leq \binom{\operatorname{tr.deg}(K)+2^{m-1}}{2}$ by Remark \ref{prank}, and so the symbol length of the class of $\varphi$ in $H_2^3(F)$ is $\leq \binom{\operatorname{tr.deg}(K)+2^{m-1}}{2}$.
\end{proof}

\section{Linked Fields}

In this section we interpret the property $sl^n_2(F)=1$ in terms of linkage. We denote the set of quadratic $n$-fold Pfister forms by $P_n(F)$. Two quadratic Pfister forms $\varphi$ and $\psi$ are ``separably $k$-linked" if there exists $\phi \in P_k(F)$ such that $\varphi=\varphi' \otimes \phi$ and $\psi=\psi' \otimes \phi$ for some bilinear Pfister forms $\varphi'$ and $\psi'$. They are ``inseparably $k$-linked" if there exists a $k$-fold bilinear Pfister form $\phi$ such that $\varphi=\phi \otimes \varphi'$ and $\psi=\phi \otimes \psi'$ for some quadratic Pfister forms $\varphi'$ and $\psi'$.
Inseparable $k$-linkage implies separable $k$-linkage but not vice versa (see \cite[Corollaire 2.1.4]{Faivre:thesis}).
We say that $I_q^n F$ is ``separably (inseparably) linked" if every two quadratic $n$-fold Pfister forms over $F$ are separably (inseparably) $(n-1)$-linked. A field $F$ with separably linked $I_q^2 F$ is known in the literature as a ``linked" field (e.g. \cite{ElmanLam:1973}).

\begin{prop}\label{corres}
For a field $F$, $I_q^n F$ is separably linked if and only if $sl_2^n(F)=1$.
\end{prop}

\begin{proof}
Suppose $I_q^n F$ is separably linked.
This means that for every two quadratic $n$-fold Pfister forms, $\varphi$ and $\psi$, there exists $\rho \in P_{n-1}(F)$ and $\alpha,\beta \in F^\times$ such that $\varphi=\langle \langle \alpha \rangle \rangle \otimes \rho$ and $\psi=\langle \langle \beta \rangle \rangle \otimes \rho$.
Then $\varphi \perp \psi \cong \langle \langle \alpha \beta \rangle \rangle \otimes \rho \pmod{I_q^{n+1} F}$. This way every orthogonal sum of several $n$-fold Pfister forms is congruent modulo $I_q^{n+1} F$ to one $n$-fold Pfister form, which means $sl_2^n(F)=1$.

In the opposite direction, assume $sl_2^n(F)=1$.
Consider two $n$-fold Pfister forms $\varphi$ and $\psi$.
Since $sl_2^n(F)=1$, there exists some $\pi \in P_n(F)$ such that $\varphi \perp \psi \cong \pi \pmod{I_q^{n+1} F}$.
By \cite[Proposition 24.5]{EKM}, there exists $\rho \in P_{n-1}(F)$ and $\alpha,\beta \in F^\times$ such that $\varphi=\langle \langle \alpha \rangle \rangle \otimes \rho$, $\psi=\langle \langle \beta \rangle \rangle \otimes \rho$ and $\pi=\langle \langle \alpha \beta \rangle \rangle \otimes \rho$.
Therefore $\varphi$ and $\psi$ are separably $(n-1)$-linked.
Consequently $I_q^n F$ is separably linked.
\end{proof}

Separable $n$-linkage of $F$ is therefore a hands-on way of understanding the property $sl_2^n(F)=1$. We obtain the fact that $(I_q^n F \neq 0) \& (sl_2^n(F)=1) \Rightarrow u^n(F) \in \{2^n,2^{n+1}\}$ for any $n \geq 2$ (Theorem \ref{Theoremu}) using this language. Recall that the Witt index $i_W(\varphi)$ of a quadratic form $\varphi : V \rightarrow F$ is the maximal dimension of a subspace $W \subseteq V$ for which $\varphi(w)=0$ for all $w \in W$. We cite here a few necessary results from the literature and prove Proposition \ref{lift} which is an essential ingredient in the proof of Theorem \ref{Theoremu}.

\begin{lem}[{\cite[Lemme 3.3.4]{Faivre:thesis}}]\label{Faivre}
Suppose $F$ is a linked field with $\operatorname{char}(F)=2$ and that $I_q^4(F)=0$. Then every two 3-fold Pfister forms over $F$ are inseparably 2-linked.
\end{lem}

\begin{thm}[{\cite[Theorem 5.6]{ChapmanDolphin:2017}}]\label{Linkage2}
If $F$ is a linked field with $\operatorname{char}(F)=2$ then $I_q^4 F=0$.
\end{thm}

\begin{prop}\label{VanishingI^n}
For $n \geq 2$, if $I_q^n F$ is separably linked then $I_q^{n+2} F=0$. If $I_q^n F$ is inseparably linked then $I_q^{n+1} F=0$.
\end{prop}

\begin{proof}
The statement appears in \cite[Corollaire 3.3.3]{Faivre:thesis} for all the cases but the case of $I_q^2 F$ separably linked. The latter is Theorem \ref{Linkage2}.
\end{proof}

\begin{prop}[{\cite[Corollary 5.4]{ChapmanGilatVishne:2017}}]\label{septoinsep}
Suppose $n \geq 2$, $\operatorname{char}(F)=2$ and $I_q^{n+1} F=0$. Then two forms in $P_n(F)$ are separably $(n-1)$-linked if and only if they are inseparably $(n-1)$-linked.
\end{prop}

\begin{prop}[{\cite[Corollaire 2.3.3]{Faivre:thesis}} or {\cite[Theorem 24.2]{EKM}}]\label{Wittindex}
Given $\varphi \in P_m(F)$ and $\psi \in P_n(F)$, $i_W(\varphi \perp \psi)=2^r$ where $r$ is the maximal integer for which $\varphi$ and $\psi$ are separably $r$-linked.
\end{prop}

\begin{prop}\label{lift}
For any $n\geq 2$, if $I_q^n F$ is separably linked then $I_q^{n+1} F$ is inseparably linked (and in particular $I_q^{n+1} F$ is also separably linked).
\end{prop}

\begin{proof}
The case of $n=2$ follows from Theorem \ref{Linkage2} and Lemma \ref{Faivre}.
Assume $n \geq 3$.
Then consider $\varphi,\psi \in P_{n+1}(F)$.
They can be written as $\varphi=\langle \langle \beta_\varphi \rangle \rangle \otimes \varphi'$ and $\psi=\langle \langle \beta_\psi \rangle \rangle \otimes \psi'$ for some $\beta_\varphi, \beta_\psi \in F^\times$ and $\varphi',\psi' \in P_n(F)$. Since $I_q^n F$ is separably linked, we can write $\varphi'=\langle \langle \gamma_\varphi \rangle \rangle \otimes \pi$ and $\psi'=\langle \langle \gamma_\psi \rangle \rangle \otimes \pi$ for some $\gamma_\varphi,\gamma_\psi \in F^\times$ and $\pi \in P_{n-1}(F)$. Write $\pi=\langle \langle \delta \rangle \rangle \otimes \pi'$ for some $\delta \in F^\times$ and $\pi' \in P_{n-2}(F)$. Consider the forms $\langle \langle \beta_\varphi,\gamma_\varphi \rangle \rangle \otimes \pi'$ and $\langle \langle \beta_\psi,\gamma_\psi \rangle \rangle \otimes \pi'$. Both are in $P_n(F)$, and since $I_q^n F$ is separably linked, they can be written as $\langle \langle \alpha_\varphi \rangle \rangle \otimes \rho$ and $\langle \langle \alpha_\psi \rangle \rangle \otimes \rho$ for some $\alpha_\varphi,\alpha_\psi \in F^\times$ and $\rho \in P_{n-1}(F)$.
Consequently, $\varphi=\langle \langle \alpha_\varphi,\delta \rangle \rangle \otimes \rho$ and $\psi=\langle \langle \alpha_\psi,\delta \rangle \rangle \otimes \rho$. Therefore $\varphi$ and $\psi$ are separably $n$-linked. 
Since $I_q^n F$ is separably linked, $I_q^{n+2} F=0$ by Proposition \ref{VanishingI^n}.
Since $\varphi$ and $\psi$ are separably $n$-linked and $I_q^{n+2} F=0$, $\varphi$ and $\psi$ are inseparably $n$-linked by Proposition \ref{septoinsep}.
\end{proof}

\section{Linked fields and $u^n(F)$}

In this section we show the restriction that having $sl_2^n(F)=1$ imposes on $u^n(F)$.

\begin{thm}\label{Theoremu}
Suppose $n \geq 2$, $\operatorname{char}(F)=2$, $I_q^n F \neq 0$ and $I_q^n F$ is separably linked. Then $u^n(F)$ is either $2^n$ or $2^{n+1}$. Moreover, every anisotropic form $\varphi \in I_q^n F$ of dimension $u^n(F)$ is Witt equivalent to $\pi \perp \psi$ where $\pi \in P_n(F)$ and $\psi \in P_{n+1}(F)$.
\end{thm}

\begin{proof}
Let $\varphi$ be an anisotropic form in $I_q^n F$.
Every form in $I_q^n F$ is Witt equivalent to a sum of scalar multiples of $n$-fold Pfister forms. Since every scalar multiple of an $n$-fold Pfister form is congruent modulo $I_q^{n+1} F$ to the $n$-fold Pfister form, $\varphi$ can be written as the sum of $n$-fold Pfister forms plus something from $I_q^{n+1} F$.
Since $I_q^n F$ is separably linked, $\varphi$ is congruent modulo $I_q^{n+1} F$ to some $\pi \in P_n(F)$ (see Proposition \ref{corres}).
Therefore $\varphi=\pi \perp \psi$ for some form $\psi \in I_q^{n+1} F$.
Now, $\psi$ is the sum of scalar multiples of forms in $P_{n+1}(F)$. But $I_q^{n+2} F=0$, so every scalar multiple of a form in $P_{n+1}(F)$ is actually in $P_{n+1}(F)$. Therefore $\psi$ is the sum of forms in $P_{n+1}(F)$.
But $I_q^{n+1} F$ is inseparably linked, so the sum of every two $(n+1)$-fold Pfister forms is Witt equivalent to an $(n+1)$-fold Pfister form (see \cite[Proposition 3.6]{ChapmanGilatVishne:2017}), which means that $\psi \in P_{n+1}(F)$.
Since $I_q^n F$ is separably linked, $\pi$ and $\psi$ are separably $(n-1)$-linked. If they are $n$-linked (which means that $\pi$ is a subform of $\psi$) then the dimension of $\varphi$ is $2^n$. Otherwise it is $2^{n+1}$ by Proposition \ref{Wittindex}.
\end{proof}

\begin{cor}\label{coru}
Suppose $\operatorname{char}(F)=2$, $I_q^n F \neq 0$ and $n \geq 2$. Then $I_q^n F$ is inseparably linked if and only if $u^n(F)=2^n$.
\end{cor}

\begin{proof}
If $I_q^n F$ is inseparably linked then it is separably linked and so $u^n(F)=2^n$ or $2^{n+1}$. Every form $\varphi \in I_q^n F$ of dimension $u^n(F)$ is Witt equivalent to $\pi \perp \psi$ where $\pi \in P_n(F)$ and $\psi \in P_{n+1}(F)$. Since $I_q^n F$ is inseparably linked, $I_q^{n+1} F=0$ by Proposition \ref{VanishingI^n}, and so $\varphi=\pi$ and $u^n(F)=2^n$.
In the opposite direction, if $u^n(F)=2^n$ then the Witt index of the sum of any two $n$-fold Pfister forms is at least $2^n$, which by Proposition \ref{Wittindex} means that $I_q^n F$ is separably linked. Since $u^n(F)=2^n$, we have $I_q^{n+1} F=0$, and so by Propositions \ref{VanishingI^n} and \ref{septoinsep}, $I_q^n F$ is inseparably linked.
\end{proof}

\begin{lem}\label{WittindexLem}
Suppose $\operatorname{char}(F)=2$, $n \geq 2$ and $I_q^n F$ is separably linked. Let $\pi \in P_n(F)$ and $\psi \in P_{n+1}(F)$.
Then $i_W(\psi \perp \pi \perp \langle 1 \rangle) \geq 2^{n-1}+1$.
\end{lem}

\begin{proof}
Since $I_q^n F$ is separably linked, $\pi=\langle \langle \alpha \rangle \rangle \otimes \rho$ and $\psi=\langle \langle \beta,\gamma \rangle \rangle \otimes \rho$ for some $\alpha,\beta,\gamma \in F^\times$ and $\rho \in P_{n-1}(F)$. Then $\psi \perp \pi \perp \langle 1 \rangle$ is isometric to $2^{n-1} \times \varmathbb{H} \perp \langle \alpha,\beta,\gamma,\beta \gamma \rangle \otimes \rho \perp \langle 1 \rangle$. The form $\langle \alpha,\beta,\gamma,\beta \gamma \rangle \otimes \rho \perp \langle 1 \rangle$ is a Pfister neighbor of $\langle \langle \alpha,\beta,\gamma \rangle \rangle \otimes \rho \in P_{n+2}(F)$. The latter is hyperbolic because $I_q^{n+2} F=0$ by Proposition \ref{VanishingI^n}, and therefore $\langle \alpha,\beta,\gamma,\beta \gamma \rangle \otimes \rho \perp \langle 1 \rangle$ is isotropic, which means $i_W(\psi \perp \pi \perp \langle 1 \rangle) \geq 2^{n-1}+1$.
\end{proof}

\begin{rem}
By \cite[Th\`{e}or\'{e}me 2.5.5]{Faivre:thesis}, if two quadratic Pfister forms $\pi$ and $\psi$ satisfy $i_W(\pi \perp \psi \perp \langle 1 \rangle) \geq 2^{n-1}+1$ then $\pi$ and $\psi$ are inseparably $(n-1)$-linked.
As a result, by Lemma \ref{WittindexLem}, if $\operatorname{char}(F)=2$, $n \geq 2$ and $I_q^n F$ is separably linked, then for any $\pi \in P_n(F)$ and $\psi \in P_{n+1}(F)$, $\pi$ and $\psi$ are inseparably $(n-1)$-linked.
\end{rem}

\begin{thm}\label{Theoremd}
Suppose $\operatorname{char}(F)=2$, $I_q^n F \neq 0$ and $n \geq 2$ץ Let $d$ be the maximal dimension of an anisotropic quadratic form Witt equivalent to some $\varphi \perp [1,\alpha]$ where $\varphi \in I_q^n F$ and $\alpha \in F$.
If $I_q^n F$ is separably linked then $d=u^n(F)$.
\end{thm}

\begin{proof}
Clearly $u^n(F) \leq d \leq u^n(F)+2$.
When $u^n(F)=2^n$, $\varphi$ is a scalar multiple of an $n$-fold Pfister form, and since $I_q^{n+1} F=0$, $\varphi$ is universal. Therefore $i_W(\varphi \perp [1,\alpha]) \geq 1$, and $d \leq 2^n=u^n(F)$.

Assume $u^n(F)=2^{n+1}$. Assume the contrary, that $d=u^n(F)+2$. Then there exists an anisotropic form $\tau=\varphi \perp [1,\alpha]$ where $\varphi$ is of dimension $2^{n+1}$ in $I_q^n F$. Recall that $\varphi$ is Witt equivalent to $\pi \perp \psi$ for some $\pi \in P_n(F)$ and $\psi \in P_{n+1}(F)$. Since $i_W(\pi \perp \psi \perp \langle 1 \rangle) \geq 2^{n-1}+1$ by Lemma \ref{WittindexLem}, $\varphi \perp \langle 1 \rangle$ is isotropic. However, the latter is a subform of $\varphi \perp [1,\alpha]$, contradictory to the assumption that $\tau$ is anisotropic.
\end{proof}

When $n=2$, the number $d$ appearing in Theorem \ref{Theoremd} is in fact $u(F)$.
Then Theorem \ref{Theoremd} states that if $F$ is a linked field of $\operatorname{char}(F)=2$ then $u^2(F)=u(F)$.
Plugging that in Theorem \ref{Theoremu} gives that if $F$ is a linked field and $I_q^2 F \neq 0$ then $u(F)$ is either 4 or 8.
This recovers \cite[Corollary 5.2]{ChapmanDolphin:2017}.
This is the characteristic 2 analogue of \cite[Main Theorem]{ElmanLam:1973} which states that if $F$ is a linked field of $\operatorname{char}(F) \neq 2$ then $u(F)$ is either $0,1,2,4$ or 8.
Plugging $n=2$ in Corollary \ref{coru} gives that when $I_q^2 F$ is separably linked, $I_q^2 F$ is inseparably linked if and only if $u(F)=4$.
This is exactly \cite[Theorem 3.1]{Baeza:1982}.

When $n=3$, the number $d$ appearing in Theorem \ref{Theoremd} is actually the invariant $cu(F)$ defined in \cite{Kashima:1996} to be the maximal dimension of a nonsingular form with trivial Clifford invariant.
Therefore, by \ref{Theoremd}, if $I_q^3 F$ is a separably linked field, $\operatorname{char}(F)=2$ and $I_q^3 F \neq 0$ then $cu(F)=8$ or 16, and $cu(F)=8$ if and only if $I_q^3 F$ is inseparably linked.
This provides a characteristic 2 analogue to \cite[Theorem 1.1 (1)]{Kashima:1996}, which states that given a field $F$ of $\operatorname{char}(F) \neq 2$, if $I_q^3 F$ is linked then $cu(F)$ is 1,2,8 or 16.

\section{Acknowledgments}
The authors thank Adrian Wadsworth for the helpful comments on the manuscript.
The first author was visiting Perimeter Institute of Theoretical Physics during the Summer of 2017, during which most of the work on this project was carried out.

\bibliographystyle{abbrv}
\bibliography{bibfile}

\end{document}